\documentclass[12pt,letterpaper,twoside]{article}

\usepackage{amsmath,amssymb}
\usepackage{amsthm}
\usepackage{stmaryrd}
\usepackage{fancyhdr}
\usepackage{times}
\usepackage{mathrsfs} 
\usepackage{titlesec}
\usepackage[usenames]{color}
\usepackage[pdftex,dvips]{geometry}

\input dvipsnam.def

\input diagxy

\newtheoremstyle{fact}
     {\topsep}
     {\topsep}
     {\slshape}
     {}
     {\bfseries}
     {}
     { }
     {\thmname{#1}\thmnumber{ #2.}\thmnote{ \rm (#3)}}

\newtheorem{theorem}{Theorem}[section]
\newtheorem{Ltheorem}{Theorem}

\newtheorem*{theorem*}{Theorem}
\newtheorem{lemma}[theorem]{Lemma}
\newtheorem{proposition}[theorem]{Proposition}
\newtheorem{corollary}[theorem]{Corollary}
\newtheorem{problem}{Problem}
\newtheorem*{problem*}{Problem}

\theoremstyle{definition}
\newtheorem{definition}[theorem]{Definition}
\newtheorem{remark}[theorem]{Remark}

\newtheorem*{remark*}{Remark}
\newtheorem{discussion}[theorem]{Discussion}

\newtheorem*{question*}{Question}
\newtheorem*{examples*}{Examples}

\newtheorem*{example*}{Example}

\newtheorem*{convention*}{Convention}

\theoremstyle{fact}

\newtheorem{ftheorem}[theorem]{Theorem}
\newtheorem{flemma}[theorem]{Lemma}

\newenvironment{myromanlist}[1][enumi]{\begin{list}{{\rm (\roman{#1})}}
{\usecounter{#1}\setlength{\labelwidth}{25pt}\setlength{\topsep}{-6pt}
\setlength{\itemsep}{-4pt} \setlength{\leftmargin}{25pt}}}{\end{list}}

\newenvironment{myalphlist}[1][enumi]{\begin{list}{{\rm (\alph{#1})}}
{\usecounter{#1}\setlength{\labelwidth}{25pt}\setlength{\topsep}{-6pt}
\setlength{\itemsep}{-4pt} \setlength{\leftmargin}{25pt}}}{\end{list}}

\def\proofont{\fontseries{bx}\fontshape{sc}\selectfont}
\def\proofname{Proof.}

\newcommand{\Note}[1]{}

\makeatletter
\renewenvironment{proof}[1][\proofname]{\par
  \normalfont
  \topsep6\p@\@plus6\p@ \trivlist
  \item[\hskip\labelsep\noindent\proofont #1]\ignorespaces
}{%
  \qed\endtrivlist
}
\makeatother

\titlelabel{\thetitle.\ }
\titleformat*{\section}{\normalsize\bfseries\centering}
\titleformat*{\subsection}{\normalsize\bfseries}
\titlespacing{\subsection}{0pt}{\topsep}{0.5ex}
\titleformat{\subsection}[runin]{\normalfont\bfseries}{%
\thesubsection.}{0.5ex}{}[.]

\author{A. P. Nguyen\thanks{I am grateful for the financial support of
the Faculty of Science at the University of Manitoba that enabled me to
carry out this research.}}

\title{Which infinite abelian groups admit an almost\\
maximally almost-periodic group topology?
\thanks{\endgraf
\hspace{5.5pt}
{\em Keywords:} almost maximally almost-periodic, Pr\"ufer group,
von Neumann radical, $T$-sequence.}}

\begin{document}

\makeatletter
\def\@fnsymbol#1{\ifcase#1\or * \or 1 \or 2  \else\@ctrerr\fi\relax}

\let\mytitle\@title
\chead{\small\itshape A. P. Nguyen / Almost maximally almost-periodic
group topologies}
\fancyhead[RO,LE]{\small \thepage}
\makeatother

\maketitle

\def\thanks#1{}

\thispagestyle{empty}

\begin{abstract}
A topological group $G$ is said to be {\em almost maximally
almost-periodic} if its von Neumann radical $\mathbf{n}(G)$ is non-trivial,
but finite. In this paper, we prove that every abelian group with an
infinite torsion subgroup admits a (Hausdorff) almost maximally
almost-periodic group topology.
Some open problems are also formulated.
\end{abstract}

\section{Introduction}

\label{sect:intro}

Every topological group $G$
admits a ``largest" compact Hausdorff group $bG$ and a continuous
homomorphism $\rho_G \colon G \rightarrow bG$ such that
every continuous homomorphism $\varphi\colon G\rightarrow K$
into a compact Hausdorff group $K$ factors uniquely through
$\rho_G$:
\begin{align}
\bfig
\Vtriangle(0,0)/->`->`<--/<300,350>%
[G`K`bG;\varphi`\rho_G`\exists!\tilde\varphi]
\efig
\end{align}
The group $bG$ is called the {\em Bohr-compactification} of $G$, and
the image $\rho_G(G)$ is dense in $bG$. The kernel of
$\rho_G$ is called the {\em von Neumann radical} of $G$, and
is denoted by $\mathbf{n}(G)$. One says that $G$ is
{\em maximally almost-periodic} if $\mathbf{n}(G)=1$, and
{\em minimally almost-periodic} if $\mathbf{n}(G)=G$ (cf.~\cite{NeuWig}).

It is well known that the discrete topology is maximally almost-periodic
on every abelian group (cf.~\cite[4.23]{HewRos}).
Ajtai, Havas, and Koml\'os, and independently, Zelenyuk and Protasov,
showed that every infinite abelian group admits a (Hausdorff) group
topology that  is not maximally almost-periodic (cf.~\cite{AHK}
and~\cite[Theorem~16]{ZelProt}). While these results provide a group
topology where the von Neumann radical is non-trivial, they
remain silent about the size of the von Neumann radical of the group. In
particular, they do not guarantee that the von Neumann radical is finite.
Motivated by these observations,
Luk\'acs called a Hausdorff topological group $G$  {\em almost maximally
almost-periodic} if $\mathbf{n}(G)$ is non-trivial, but finite
(cf.~\cite{GL9}). He proved, among other results,  that for every prime
$p\neq 2$, the Pr\"ufer group $\mathbb{Z}(p^\infty)$ admits a
(Hausdorff) almost maximally almost-periodic group topology
(cf.~\cite[4.4]{GL9}).

The aim of this paper is to substantially extend the results of Luk\'acs
in several directions. The main results of the paper are as follows:

\begin{Ltheorem} \label{main:ctbl}
Let $A$ be an abelian  group with an infinite torsion subgroup. Then
$A$ admits a (Hausdorff) almost maximally almost-periodic group topology.
\end{Ltheorem}

\begin{Ltheorem} \label{main:p2}
Let $p$ be a prime, and
$x \in \mathbb{Z}(p^\infty)$ a non-zero element. Then there is a
(Hausdorff) group topology $\tau$ on $\mathbb{Z}(p^\infty)$ such that
$\mathbf{n}(\mathbb{Z}(p^\infty),\tau)=\langle x\rangle$.
\end{Ltheorem}

Most of the effort in this paper is put toward proving
Theorem~\ref{main:p2}, which implies Theorem~\ref{main:ctbl}. Once
Theorem~\ref{main:p2} has been established, Theorem~\ref{main:ctbl}
follows from it and from another result of Luk\'acs (cf.~\cite[3.1]{GL9}).
Since Theorem~\ref{main:p2} was proven by Luk\'acs for all primes $p >2$
(cf.~\cite[4.4]{GL9}), it remains to be shown that the statement also
holds for $p=2$.

The paper is structured as follows: In order to make the manuscript more
self-contained, in Section~\ref{sect:prelim}, we have collected some
preliminary results and techniques that will be used throughout the paper.
Section~\ref{sect:canon} is a somewhat technical preparation for the proof
of Theorem~\ref{main:p2}, which is presented in Section~\ref{sect:p2}
along with the proof of Theorem~\ref{main:ctbl}. Finally, in
Section~\ref{sect:open}, we formulate two open problems stemming from the
results presented in this paper, and discuss what is known to us, at this
point, about their solution.

\section{Preliminaries}

\label{sect:prelim}

In this section, we have collected some preliminary results and
techniques that are  used throughout the paper. Thus, the experienced
or expert reader may wish to skip this section.

In this paper, all groups are abelian, and all group topologies are
Hausdorff, unless otherwise stated.
For a topological group $A$, let $\hat A =
\mathscr{H}(A,\mathbb{T})$ denote the Pontryagin dual of $A$---in
other words, the group of {\em continuous characters} of $A$ (i.e.,
continuous homomorphisms $\chi\colon A\rightarrow\mathbb{T}$, where
$\mathbb{T}=\mathbb{R}/\mathbb{Z}$) equipped with the compact-open
topology. It follows from the famous Peter-Weyl Theorem
(\cite[Thm.~33]{Pontr}) that the Bohr-compactification of $A$ can be
quite easily computed: $bA = \widehat{\hat A_d}$, where $\hat A_d$
stands for the group $\hat A$ with the discrete topology. Thus,
\begin{align} \label{eq:nA:ker}
\mathbf{n}(A)=\bigcap\limits_{\chi\in\hat A} \ker \chi.
\end{align}

The group $\mathbb{Z}(p^\infty)$
can be seen as the subgroup of $\mathbb{Q}/\mathbb{Z}$ generated by
elements of $p$-power order, or as the group formed by all $p^n$-th roots
of unity in $\mathbb{C}$.
Throughout this note, the additive notation
provided by $\mathbb{Q}/\mathbb{Z}$ is used, and we set
\mbox{$e_n= \frac 1{p^n}+\mathbb{Z}$}. The Pontryagin dual
$\widehat{\mathbb{Z}(p^\infty)}$  of $\mathbb{Z}(p^\infty)$ is the
$p$-adic group $\mathbb{Z}_p$. We let $\chi_1$  denote the
natural embedding of $\mathbb{Z}(p^\infty)$ into $\mathbb{T}$.
Luk\'acs, who proved Theorem~\ref{main:p2} for \mbox{$p\neq 2$}
(cf.~\cite[4.4]{GL9}), used so-called $T$-sequences as his main machinery
to produce almost maximally almost-periodic group topologies on
$\mathbb{Z}(p^\infty)$. While the outstanding case of \mbox{$p=2$} requires
special attention, the techniques used in this paper are nevertheless
similar.

A sequence $\{a_n\}$ in a group~$G$ is a {\em $T$-sequence} if there
is a Hausdorff group topology $\tau$ on $G$ such that  $a_n\stackrel
\tau \longrightarrow e$. In this case, the group $G$ equipped with
the finest group topology with this property is denoted by
$G\{a_n\}$. The notion of a $T$-sequence was introduced and
extensively investigated by Zelenyuk and Protasov, who characterized
$T$-sequences (and so-called $T$-filters), and studied the
topological  properties of $G\{a_n\}$, where $\{a_n\}$ is a
$T$-sequence (cf.~\cite[Theorems 1-2]{ZelProt} and~\cite[2.1.3,
2.1.4, 3.1.4]{PZMono}). These two authors used the technique of
$T$-sequences to prove the following results (some of which were also
obtained independently by Ajtai, Havas, and Koml\'os~\cite{AHK}).

\begin{ftheorem} \label{thm:prel:ZP} \mbox{ }
\begin{myalphlist}

\item
{\rm ({\cite[\S 2]{AHK}}, {\cite[Example~4]{ZelProt}})}
$\mathbb{Z}$ admits a minimally almost-periodic group
topology.

\item
{\rm (\cite[\S 4]{AHK}, \cite[Example~6]{ZelProt})}
$\mathbb{Z}(p^\infty)$ admits a minimally almost-periodic group
topology for every prime $p$.

\item
{\rm (\cite[Example~6]{ZelProt}, \cite[3.3]{DikMilTon})}
Let $\chi\in \widehat{\mathbb{Z}(p^\infty)} = \mathbb{Z}_p$.
One has
$\chi(e_n) \longrightarrow 0$ if and only if there is $m \in \mathbb{Z}$
such that $\chi = m \chi_1$.

\end{myalphlist}
\end{ftheorem}

Since $\mathbb{Z}(2^\infty)$ is an abelian group, we need only the abelian
version of the Zelenyuk-Protasov criterion:

\begin{ftheorem}[{\cite[2.1.4]{PZMono}, \cite[Theorem~2]{ZelProt}}]
\label{thm:ZelProt}
Let $\underline a = \{a_k\}$ be a sequence in an abelian group $A$.
For $l,m\in \mathbb{N}$, put
\begin{align}
A(l,m)_{\underline a}=\{ m_1 a_{k_1} + \cdots + m_h a_{k_h} \mid
m\leq k_1<\cdots <k_h,m_i\in\mathbb{Z}\backslash\{0\},\sum |m_i| \leq l\}.
\end{align}
Then  $\{a_k\}$ is a $T$-sequence if and  only if for every
$l\in\mathbb{N}$ and $g\neq 0$, there exists $m \in \mathbb{N}$ such that
$g \not\in A(l,m)_{\underline a}$.
\end{ftheorem}

For a group $A$, we put $A[n]=\{a\in A \mid  na=0\}$ for every
$n \in \mathbb{N}$. The group
$A$ is {\it almost torsion-free} if $A[n]$ is finite for every
$n \in \mathbb{N}$ (cf. \cite{TkaYasch}). Clearly, the Pr\"ufer groups
$\mathbb{Z}(p^\infty)$ are
almost torsion-free. Luk\'acs characterized
$T$-sequences in almost torsion-free groups as follows.

\begin{ftheorem}[{\cite[2.2]{GL9}}]
\label{thm:Lukacs}
Let $A$ be an almost torsion-free group, and let
$\underline a = \{a_k\}$ be a sequence in~$A$.
The following statements are equivalent:

\begin{myromanlist}

\item
For every $l,n \in \mathbb{N}$, there exists $m_0\in \mathbb{N}$ such
that $A[n] \cap A(l,m)_{\underline a}=\{0\}$ for every $m \geq m_0$.

\item
$\{a_k\}$ is a $T$-sequence.

\end{myromanlist}
\end{ftheorem}

Luk\'acs also provided sufficient conditions for a sequence in
$\mathbb{Z}(p^\infty)$ to be a $T$-sequence.

\begin{flemma}[{\cite[4.1]{GL9}}]
\label{lem:Tseq} Let $\{a_k\}$ be a sequence in
$\mathbb{Z}(p^\infty)$ such that $o(a_k) = p^{n_k}$. If $n_{k+1} -
n_k \longrightarrow \infty$, then $\{a_k\}$ is a $T$-sequence.
\end{flemma}

It turns out that the class of abelian groups that admit an almost
maximally almost-periodic group topology is upward closed in the
sense that if a group $A$ belongs there, then so does every abelian
group $B$ that contains $A$ as a subgroup
(cf. Theorem~\ref{thm:proofs:absorb}). Thus, in the proof of
Theorem~\ref{main:ctbl}, we restrict our attention to torsion
groups. In particular, we rely on the following result on the
structure of infinite abelian groups to confine our attention
further to two special subgroups.
\begin{flemma}[{\cite[Theorems 8.4, 23.1, 27.2]{Fuchs}}]
\label{lem:infinite} Every infinite abelian group contains a
subgroup that is isomorphic to $\mathbb{Z}$, or
$\mathbb{Z}(p^\infty)$, or an infinite direct sum of non-trivial
finite cyclic groups.
\end{flemma}

It follows from the above lemma that in the proof of
Theorem~\ref{main:ctbl}, we can focus on the following two types of
subgroups: Pr\"ufer groups (which are taken care of by
Theorem~\ref{main:p2}), and direct sums of infinitely many finite
groups, which are addressed by the following result.

\begin{ftheorem}[{\cite[3.1]{GL9}}] \label{thm:prel:dsf}
If $A$ is a direct sum of infinitely many non-trivial finite abelian
groups, then $A$ admits an almost maximally almost-periodic group
topology.
\end{ftheorem}

\section{The canonical form in $\boldsymbol{\mathbb{Z}(2^\infty)}$}

\label{sect:canon}

The aforementioned
result of Luk\'acs concerning $\mathbb{Z}(p^\infty)$ is based on a
{\em canonical form}, which he introduced for writing each element of
$\mathbb{Z}(p^\infty)$ uniquely as an integer combination of the
elements $\{e_n\}_{n \in \mathbb{N}}$ with certain additional
conditions on the coefficients (cf.~\cite[4.6]{GL9}). The canonical
form of Luk\'acs, however, requires the prime $p$ to be odd, and
thus fails in the case of $p=2$. In  this section, we remedy this,
provide a unique canonical form for elements of $\mathbb{Z}(2^\infty)$,
and establish some technical properties of the canonical form that are
needed for the proof of Theorem~\ref{main:p2}.

Each element \mbox{$y \in \mathbb{Z}(2^\infty)$} admits many
representations of the form $y=\sum \sigma_n e_n$, where
\mbox{$\sigma_n \hspace{-0.25pt}\in\hspace{-0.25pt} \mathbb{Z}$} with
only finitely many of the $\sigma_n$ being non-zero.
In order to find a canonical form for these elements, we  first eliminate
the summands with odd indices.

\begin{lemma}\label{lemma:canon:even}
Let  $y=\sum\sigma_{n}e_{n} \in \mathbb{Z}(2^{\infty})$. Then $y$
can be represented in the form of $y =\sum\sigma^\prime_{2n}e_{2n}$,
where $\sigma^\prime_{2n} \in \mathbb{N}$ and
$\sum|\sigma^\prime_{2n}| \leq 2\sum|\sigma_{n}|$.
\end{lemma}

\begin{proof}
Let $K$ be the largest index such that $\sigma_{K} \neq 0$,  and $N$ the
smallest integer that satisfies $K \leq 2N$. Since $e_{2n-1} = 2e_{2n}$
for every $n \in \mathbb{N}$, one has
\begin{align}
y & =
\sum\limits_{n=1}^{2N} \sigma_{n} e_n =
\sum\limits_{n=1}^{N} \sigma_{2n-1} e_{2n-1}  +
\sum\limits_{n=1}^{N} \sigma_{2n} e_{2n}\\
&  = \sum\limits_{n=1}^{N} 2\sigma_{2n-1} e_{2n} +
\sum\limits_{n=1}^{N} \sigma_{2n} e_{2n}
 = \sum\limits_{n=1}^{N} (2\sigma_{2n-1} + \sigma_{2n}) e_{2n}.
\intertext{Thus, by setting $\sigma^\prime_{2n}=2\sigma_{2n-1} + \sigma_{2n}$
for every $n \in \mathbb{N}$, one obtains
$y=\sum\sigma^\prime_{2n}e_{2n}$ and}
\sum\limits_{n=1}^{N}|\sigma^\prime_{2n}|& =
\sum\limits_{n=1}^{N}|2\sigma_{2n-1} + \sigma_{2n}| \leq
\sum\limits_{n=1}^{N}(|2\sigma_{2n-1}| + |\sigma_{2n}|)  \\
&\leq \sum\limits_{n=1}^{N}(2|\sigma_{2n-1}| + 2|\sigma_{2n}|) =
2\sum\limits_{n=1}^{2N}|\sigma_n|,
\end{align}
as desired.
\end{proof}

\begin{definition}
Let $y \in \mathbb{Z}(2^{\infty})$.
We say that $y =\sum\sigma_{2n}e_{2n}$
is the {\em canonical form} of
the element~$y$ if
$\sigma_{2n}\hspace{-2pt}\in\hspace{-2pt}
\{-1,0,1, 2\}$ for every  $n \in \mathbb{N}$
(and  $\sigma_{2n}=0$ for all but finitely many indices $n$); in this
case, we put $\Lambda(y) = \{ n\in \mathbb{N}\mid \sigma_{2n}\neq 0\}$ and
$\lambda(y)= |\Lambda(y)|$.
\end{definition}

In order for $\Lambda(y)$ and $\lambda(y)$ to be well-defined, we first
show that each \mbox{$y \in \mathbb{Z}(2^\infty)$} admits a~unique
canonical form.  We put $\underline{e} =\{e_n\}_{n=1}^{\infty}$,
and use the notation introduced in Theorem~\ref{thm:ZelProt}.

\begin{theorem} \label{thm:canon:form}
Let $y = \sum\sigma_{2n}e_{2n} \in \mathbb{Z}(2^{\infty})$. Then,

\begin{myalphlist}

\item
$y$ admits a canonical form $y=\sum\sigma^\prime_{2n}e_{2n}$ that
satisfies
\begin{align} \label{canon:eq:estimate}
\sum f(\sigma^\prime_{2n}) \leq \sum f(\sigma_{2n}),
\end{align}
where $f\colon \mathbb{R}\rightarrow \mathbb{R}$ is defined by $f(x)
= \max\{-2x,x \}$;

\item
the canonical form is unique;

\item
$\lambda(z) \leq 4l$ for every $z
\in\mathbb{Z}(2^\infty)(l,1)_{\underline e}$ and $l \in \mathbb{N}$.

\end{myalphlist}
\end{theorem}

In order to make the proof of Theorem~\ref{thm:canon:form} more
transparent, we summarize the properties of the function $f(x)$ in
the following lemma, whose easy, but nevertheless technical, proof is
omitted.

\begin{lemma} \label{lemma:canon:f}
Let $f\colon \mathbb{R} \rightarrow \mathbb{R}$ be defined by
$f(x)= \max\{-2x,x \}$. Then, for every $a,b \in \mathbb{R}$:

\begin{myalphlist}
\item
$f(a) \leq 2|a|$;

\item
$f(a) \leq |b|$ if and only if $- \frac{1}{2}|b| \leq a \leq |b|$;

\item
$f(a) \geq |b|$ if and only if $a \leq -\frac{1}{2}|b|$ or $a \geq
|b|$;

\item
$f(a + b) \leq f(a) + f(b)$;

\item
$f(a) + f(b) \leq f(a + 4b)$, provided that $a \in [-1,2]$ and $|b|
\geq 1$ or $b=0$.

\end{myalphlist}
\end{lemma}

In what follows, we also rely on the following well-known
property of $p$-groups.

\begin{remark} \label{rem:canon:ord}
Let $o(x)$ denote the order of an element $x$ in a~group.
If $P$ is a $p$-group, $a,b \in P$, and $o(a) \neq o(b)$, then
$o(a+b)=\max\{o(a),o(b)\}$.
\end{remark}

\begin{proof}[Proof of Theorem~\ref{thm:canon:form}.]
(a) Let $2N$ be the largest index such that $\sigma_{2N} \neq 0$. We
proceed by induction on $N$. If $N = 1$, then y = $\sigma_{2}e_{2}$,
and one may write $\sigma_{2} = \sigma^\prime_{2} + 4m$ with
$\sigma^\prime_{2}\in \{-1,0, 1, 2\}$ and $m \in \mathbb{Z}$. Since
$4e_2 = 0$,
\begin{align}
y = (\sigma^\prime_{2} + 4m)e_{2} = \sigma^\prime_{2}e_{2} +
m (4e_{2}) = \sigma^\prime_{2}e_{2},
\end{align}
and by Lemma~\ref{lemma:canon:f}(e), $f(\sigma^\prime_{2})+ f(m)
\leq f(\sigma_2)$. In particular, $f(\sigma^\prime_{2}) \leq
f(\sigma_2)$.

Suppose that the statement holds for all elements
with a representation where the  maximal non-zero index less than $2N$ and
$N>1$. Let $\sigma_{2N} = \sigma_{2N}^\prime + 4m$ be a division of
$\sigma_{2N}$  by $4$ with residue in $\mathbb{Z}$ such that
$\sigma^\prime_{2N}\in \{-1,0, 1, 2\}$. Since $4e_{2N} = e_{2N-2}$,
one obtains that
\begin{align}
y & = \sum\limits_{n=1}^{N}\sigma_{2n}e_{2n} =
\sum\limits_{n=1}^{N-2}\sigma_{2n}e_{2n} +
\sigma_{2N-2}e_{2N-2} + \sigma_{2N}e_{2N} \\
& = \sum\limits_{n=1}^{N-2}\sigma_{2n}e_{2n} +
\sigma_{2N-2}e_{2N-2} + (\sigma_{2N}^\prime + 4m)e_{2N} \\
& = \sum\limits_{n=1}^{N-2}\sigma_{2n}e_{2n} +
(\sigma_{2N-2}+m)e_{2N-2} + \sigma_{2N}^\prime e_{2N}.
\end{align}
The element $z =
\sum\limits_{n=1}^{N-2}\sigma_{2n}e_{2n} + (\sigma_{2N-2}
+m)e_{2N-2}$ satisfies the inductive hypothesis, and so it can be
represented in the canonical form
$z=\sum\limits_{n=1}^{N-1}\sigma^\prime_{2n}e_{2n}$, where
$\sigma_n^\prime \in \{-1,0,1,2\}$ and
\begin{align} \label{canon:eq:ind}
\sum\limits_{n=1}^{N-1} f(\sigma^\prime_{2n}) \leq
\sum\limits_{n=1}^{N-2}f(\sigma_{2n})+ f(\sigma_{2N-2} + m).
\end{align}
Thus, one has

\begin{align}
y = z + \sigma^\prime_{2N}e_{2N} =
\sum\limits_{n=1}^{N}\sigma^\prime_{2n}e_{2n}.
\end{align}
By Lemma~\ref{lemma:canon:f}(d),
\begin{align} \label{canon:eq:fsm}
f(\sigma_{2N-2} + m) \leq
f(\sigma_{2N-2}) + f(m),
\end{align}
and by Lemma~\ref{lemma:canon:f}(e),
\begin{align} \label{canon:eq:fmss}
f(m)+f(\sigma_{2N}^\prime) \leq f(\sigma_{2N}).
\end{align}
Therefore,  one obtains that
\begin{align}
\hspace{-10pt}
\sum\limits_{n=1}^{N} f(\sigma^\prime_{2n}) &
\stackrel{(\ref{canon:eq:ind})}\leq
\sum\limits_{n=1}^{N-1} f(\sigma^\prime_{2n}) + f(\sigma^\prime_{2N})
\leq \sum\limits_{n=1}^{N-2}f(\sigma_{2n})+ f(\sigma_{2N-2} + m)
+ f(\sigma^\prime_{2N}) \\
& \stackrel{(\ref{canon:eq:fsm})} \leq
\sum\limits_{n=1}^{N-2}f(\sigma_{2n})+ f(\sigma_{2N-2}) + f(m)
+ f(\sigma^\prime_{2N})
= \sum\limits_{n=1}^{N-1}f(\sigma_{2n}) + f(m) + f(\sigma^\prime_{2N})\\
& \stackrel{(\ref{canon:eq:fmss})} \leq
\sum\limits_{n=1}^{N-1}f(\sigma_{2n}) + f(\sigma_{2N})
=\sum\limits_{n=1}^{N}f(\sigma_{2n}).
\end{align}
Hence, (\ref{canon:eq:estimate}) holds for $y$, as desired.

(b) Suppose that $\sum\sigma_{2n}e_{2n}=\sum\nu_{2n}e_{2n}$ are two
distinct canonical representations of the same element in
$\mathbb{Z}(2^\infty)$. Then,  $\sum(\sigma_{2n}-\nu_{2n})e_{2n} = 0$
and $|\sigma_{2n}-\nu_{2n}| \leq 3$. Let $2N$ be the largest index
such that $ \sigma_{2N}\neq \nu_{2N}$. (Since all coefficients are
zero, except for a finite number of indices, such an $N$ exists.) This
means that $0 < | \sigma_{2N} -\nu_{2N}| \leq 3$, and so $2^{2N-1}
\leq o((\sigma_{2N}-\nu_{2N})e_{2N})$. On the other hand,
\begin{align}
o(\sum\limits_{n<N}(\sigma_{2n} - \nu_{2n})e_{2n}) \leq
\max\limits_{n < N} o((\sigma_{2n} - \nu_{2n})e_{2n}) \leq 2^{2N-2} <
o((\sigma_{2N}-\nu_{2N})e_{2N}).
\end{align}
Therefore, by Remark~\ref{rem:canon:ord},
\begin{align}
o(\sum(\sigma_{2n} - \nu_{2n})e_{2n}) & =
\max\{o(\sum\limits_{n<N}(\sigma_{2n} - \nu_{2n})e_{2n}),
o((\sigma_{2N}-\nu_{2N})e_{2N}\} \geq 2^{2N-1},
\end{align}
contrary to the assumption that $\sum(\sigma_{2n} - \nu_{2n})e_{2n}
= 0$. Hence, $\sigma_{2n} = \nu_{2n}$ for every  $ n\in \mathbb{N}$.

\pagebreak[2]

(c) Let  $z = \nu_1e_{n_1} + \ldots + \nu_te_{n_t}$, where
$\sum|\nu_i| \leq l$ and $n_1 < \ldots < \nu_t$. By
Lemma~\ref{lemma:canon:even}, $z$ can be expressed as $z =
\sum\sigma_{2n}e_{2n}$, such that
\mbox{$\sum|\sigma_{2n}| \leq 2\sum|\nu_n| \leq 2l$}.
By Lemma~\ref{lemma:canon:f}(a),
$f(\sigma_{2n}) \leq 2|\sigma_{2n}|$, so one obtains that
\mbox{$\sum f(\sigma_{2n}) \leq \ 2\sum|\sigma_{2n}| \leq 4l$}.
By (a), $z$ admits a canonical form $z =
\sum\sigma^\prime_{2n}e_{2n}$ where
\mbox{$\sum f(\sigma^\prime_{2n}) \leq \sum f(\sigma_{2n}) \leq 4l$}.
Since \mbox{$f(\sigma^\prime_{2n}) \geq 0$},
there can only be at most $4l$ indices with non-zero coefficients
$\sigma^\prime_{2n}$.
\end{proof}

\begin{lemma}\label{lemma:canon:Lambda}
Let m $\in\mathbb{Z}\backslash\{0\}$, and put $l=\lceil\log_4
|m|\rceil$. If $n >l$, then $\Lambda(me_{2n}) \subseteq \{n-l,
\ldots,n-1,n\}$ and $1\leq \lambda(me_{2n})$.
\end{lemma}

\begin{proof}
Since $n >l$, we have that $2^{2n} > |m|$, and so $me_{2n} \neq 0$.
Thus, $1\leq \lambda(me_{2n})$. One may expand $m$ in the form of
$m = \mu_{0} + \mu_{2}2^{2}+ \cdots + \mu_{2l}2^{2l}$,
where $\mu_{i} \in \{-1, 0,1, 2\}$. Therefore,
\begin{align}
me_{2n} = \mu_{0}e_{2n} + \mu_2e_{2n-2} + \cdots +\mu_{2l}e_{2n-2l}
\end{align}
is in the canonical form. Hence,
$\Lambda(me_{2n}) \subseteq \{n-l,\ldots,n-1,n\}$, as desired.
\end{proof}

\begin{lemma}\label{lemma:canon:y-z}
Let $y, z \in \mathbb{Z}(2^{\infty})$ be such that $\lambda(y)
>\lambda(z)$, and suppose that $\Lambda(y) =
\{k_1, \ldots ,k_g\}$  where $k_1 < \cdots < k_g$ and $g
=\lambda(y)$. Then, $o(y - z) >  {4}^{k_{g-\lambda(z)}-1}$.
\end{lemma}

\begin{proof}
Let $y =\sum\nu_{2n}e_{2n}$ and $z =\sum\mu_{2n}e_{2n}$ be the
canonical forms of $y$ and $z$. Since
Theorem~\ref{thm:canon:form}(b) provides that the canonical form is
unique, $\lambda(y) > \lambda(z)$ implies that $y \neq z$, and thus
$y - z \neq 0$. Let $N$ be the largest integer such that
$\nu_{2N}-\mu_{2N} \neq 0$. Then,  $\nu_{2n}=\mu_{2n}$ for every $n >
N$. In particular, $\mu_{2k_{i}} \neq 0$  for every $k_{i}> N$.
Therefore, there are at most $\lambda(z)$ many indices $k_{i}$ that
satisfy $k_{i}>N$. Hence, $N \geq k_{g-\lambda(z)}$.
Since $0 < | \nu_{2N} -\mu_{2N}| \leq 3$, one has $2^{2N-1} \leq
o((\nu_{2N}-\mu_{2N})e_{2N})$. On the other hand,
\begin{align}
o(\sum\limits_{n<N}(\nu_{2n} - \mu_{2n})e_{2n}) \leq \max\limits_{n
< N} o((\nu_{2n} - \mu_{2n})e_{2n}) \leq 2^{2N-2} <
o((\nu_{2N}-\mu_{2N})e_{2N}).
\end{align}
Consequently, by Remark~\ref{rem:canon:ord},
\begin{align}
\hspace{-6pt}
o(\sum(\nu_{2n} - \mu_{2n})e_{2n}) & =
\max\{o(\sum\limits_{n<N}(\nu_{2n} - \mu_{2n})e_{2n}),
o((\nu_{2N}-\mu_{2N})e_{2N}\} \geq 2^{2N-1} > 4^{N-1}.
\end{align}
Hence, $o(y - z)= o(\sum(\nu_{2n} - \mu_{2n})e_{2n}) >
4^{k_{g-\lambda(z)}-1}$, as desired.
\end{proof}

\begin{remark} \label{rem:canon:Lambda}
If $y_1,y_2 \in \mathbb{Z}(2^\infty)$ and $\Lambda(y_1) \cap
\Lambda(y_2) = \varnothing$, then $\Lambda(y_1 + y_2) = \Lambda(y_1)
\cup \Lambda(y_2)$ and $\lambda(y_1 + y_2) = \lambda(y_1) +
\lambda(y_2)$.
\end{remark}

\begin{proposition}\label{prop:canon:y}
Let \mbox{$y = \nu_{1}e_{2n_{1}} + \cdots + \nu_{t}e_{2n_{t}}$},
where \mbox{$\nu_i \neq 0$} and \mbox{$0< n_1< \cdots <n_t$} are
integers. Put \mbox{$l_i = \lceil \log_4|\nu_i|\rceil$,} and suppose
that $n_{i} < n_{i+1}-l_{i+1}$ for each $1 \leq i <t$. Then,
\begin{myalphlist}
\item
$t \leq \lambda(y)$;

\item
if $z \in \mathbb{Z}(2^\infty)$ is such that $\lambda(z) <
\lambda(y)$, then $o(y - z) > 4^{n_{t - \lambda(z)} - l_{t
-\lambda(z)}-1}$.

\end{myalphlist}
\end{proposition}

\begin{proof}
(a)  By Lemma~\ref{lemma:canon:Lambda},
\begin{align} \label{eq:canon:subset}
\Lambda(\nu_ie_{2n_i}) \subseteq \{n_i - l_i, \ldots ,n_i\}
\end{align}
 for each $1 \leq i \leq t$. Thus, the
sets $\Lambda(\nu_ie_{2n_i})$ are pairwise disjoint, because $n_{i}
< n_{i+1} - l_{i+1}$. Therefore, by Remark~\ref{rem:canon:Lambda}, one
obtains that $\lambda(y) = \lambda(\nu_1e_{2n_1}) + \cdots +
\lambda(\nu_te_{2n_t}) \geq t$, and
\begin{align}\label{lambda:y}
\Lambda(y)=\bigcup\limits_{i=1}^{t}\Lambda(\nu_i e_{2n_i})\subseteq
\bigcup\limits_{i=1}^t \{n_i-l_i,\ldots,n_i\}.
\end{align}

(b) Suppose that $\Lambda(y)= \{k_1, \ldots, k_g\}$ (increasingly
ordered). For any $i$ such that $t-i \leq 0$, define $n_{t-i} = n_1$
and $l_{t-i} = l_1$. We proceed by induction on $i$ to show that
$k_{g-i} \geq n_{t-i} - l_{t-i}$ for all $0 \leq i \leq g-1$.

For $i=0$, Lemma~\ref{lemma:canon:Lambda} yields that \mbox{$\Lambda(y)
\cap \Lambda(\nu_{t}e_{2n_{t}})= \Lambda(\nu_{t}e_{2n_{t}}) \neq
\emptyset$}. Since \mbox{$n_{i} < n_{i+1}-l_{i+1}$} for
each $0\leq i <t$, $\Lambda(\nu_{t}e_{2n_{t}})$ contains the largest
elements of $\bigcup\limits_{i=1}^t \{n_i-l_i,\ldots,n_i\}$, and
thus contains the largest value in $\Lambda(y)$, namely $k_g$.
Hence, $k_g \geq n_t-l_t$. Suppose that the statement holds for all
integers $i<N$. For $i=N$, if \mbox{$k_{g-N} < n_{t-N} -l_{t-N}$},
then for all \mbox{$N \leq i \leq g-1$}, \mbox{$k_{g-i} \leq k_{g-N}
< n_{t-N} -l_{t-N}$}. Moreover, by the inductive hypothesis, for all
$0 \leq i < N$, one has that $k_{g-i} \geq n_{t-i} - l_{t-i}>
n_{t-N}$, because $n_{i} < n_{i+1} - l_{i+1}$. So,
\begin{align}
\Lambda(\nu_{t-N}e_{2n_{t-N}})\cap\{n_{t-N}-l_{t-N},\ldots,n_{t-N}\}\subseteq\Lambda(y)\cap
\{n_{t-N}-l_{t-N}, \ldots , n_{t-N}\} =
\emptyset.
\end{align}
Thus, by (\ref{eq:canon:subset}), $\Lambda(\nu_{t-N}e_{2n_{t-N}}) = \emptyset$,
which contradicts \mbox{$1 \leq \lambda(\nu_{t-N}e_{2n_{t-N}})$}
from Lemma~\ref{lemma:canon:Lambda}. Hence, one has that \mbox{$k_{g-N} \geq
n_{t-N}-l_{t-N}$} for all $N$. It follows from (\ref{lambda:y}) that \mbox{$k_i
\geq n_1 - l_1$} for all $i$. So, \mbox{$k_{g-i} \geq n_{t-i} - l_{t-i}$}
holds even for $t-i \leq 0$. Thus, for $i=\lambda(z)$,
\mbox{$k_{g-\lambda(z)} \geq n_{t-\lambda(z)} -
l_{t-\lambda(z)}$}. By Lemma~\ref{lemma:canon:y-z}, \mbox{$o(y - z)>
4^{k_{\lambda(y) - \lambda(z)} -1} \geq 4^{ n_{t - \lambda(z)} -
l_{t - \lambda(z)} -1}$}, as desired.
\end{proof}

\begin{corollary}\label{cor:canon:y-z}
Let $l \in \mathbb{N}, z \in \mathbb{Z}(2^\infty)(l,1)_{\underline
e}$, and $ y = e_{2n_1} + \cdots + e_{2n_t}$ such that $n_1 <
\cdots< n_t$, $4l < t$, and $n_i < n_{i+1} - l$.  Then, $o(\mu y + z)
> 4^{n_{t - 4l} - l -1} \geq 4^{n_1 - l - 1}$ for every $\mu \in
\mathbb{Z}$ such that $0 \leq |\mu| \leq l$.
\end{corollary}

\begin{proof}
Since $|\mu| \leq l$, one has that $\log_4|\mu| < l$ and $\lceil
\log_4|\mu|\rceil \leq l$. So, \mbox{$\mu y = \mu e_{2n_1} + \cdots
+ \mu e_{2n_t}$} satisfies the conditions of
Proposition~\ref{prop:canon:y}. By Proposition~\ref{prop:canon:y}(a), one
obtains that \mbox{$4l < t \leq \lambda(\mu y)$}. Moreover, if
\mbox{$z = \nu_1e_{n_1} + \ldots + \nu_se_{n_s} \in
\mathbb{Z}(2^\infty)(l,1)_{\underline e}$,} where $n_1 < \cdots<
n_s$ and $\sum\limits_{i=1}^{s}|\nu_i| \leq l$, then
\begin{align}
-z =(-\nu_1)e_{n_1} + \ldots + (-\nu_s)e_{n_s} \in
\mathbb{Z}(2^\infty)(l,1)_{\underline e},
\end{align}
since
\mbox{$\sum\limits_{i=1}^{s}|-\nu_i| = \sum\limits_{i=1}^{s}|\nu_i|
\leq l$}. By Theorem~\ref{thm:canon:form}(c), $\lambda(-z) \leq 4l <
\lambda(\mu y)$. Thus, $\mu y$ and $-z$ satisfy the conditions of
Proposition~\ref{prop:canon:y}(b), and so one has that
\begin{align}
o(\mu y + z) = o(\mu y - (-z)) > 4^{n_{t - \lambda(z)} -
 \lceil \log_4|\mu| \rceil - 1} \geq 4^{n_{t - \lambda(z)} - l - 1} >
4^{n_{t - 4l} - l - 1} \geq 4^{n_1 - l -1},
\end{align}
as desired.
\end{proof}

\section{Proofs of Theorem A and B}

\label{sect:p2}

In this section, we first prove Theorem~\ref{main:p2}, and then we
apply this result to prove Theorem~\ref{main:ctbl}. For
Theorem~\ref{main:p2}, Luk\'acs has already established this result
for all primes~but $p=2$ (cf.~\cite[4.4]{GL9}). Thus, we confine our
attention to the Pr\"ufer group $\mathbb{Z}(2^\infty)$. Clearly,
Theorem~\ref{thm:p2:Tseq} below, together with the result of
Luk\'acs, implies Theorem~\ref{main:p2}. In
Theorem~\ref{thm:p2:Tseq}, we construct a~$T\mbox{-}$sequence in
$\mathbb{Z}(2^\infty)$, and show that its von Neumann radical
$\mathbf{n}(G)$ is a prefixed cyclic subgroup.

\begin{theorem}\label{thm:p2:Tseq}
For $x \in \mathbb{Z}(2^\infty)\backslash\{0\}$ such that
$o(x) =2^{k_0}$, put
\begin{align}
b_{k} = -x + e_{2(k^3-k^2)} + \ldots + e_{2(k^3- 2k)} + e_{2(k^3-k)}+ e_{2k^3}.
\end{align}
Consider the sequence $\{d_k\}$, defined as $b_1, e_1, b_2, e_2,
b_3, e_3, \ldots$ Then,
\begin{myalphlist}

\item
$\{d_k\}$ is a T-sequence in $\mathbb{Z}(2^\infty)$;

\item
the underlying group of $\widehat{\mathbb{Z}(2^\infty)\{d_k\}}$ is
$\langle 2^{k_0} \chi_1 \rangle$;

\item
$\mathbf{n}(\mathbb{Z}(2^\infty)\{d_k\})= \langle x\rangle$.

\end{myalphlist}
\end{theorem}

\begin{proof}
(a) To shorten the notation, we denote $A =\mathbb{Z}(2^\infty)$.

In order to show that $\{d_k\}$ is a $T$-sequence, we prove that it
satisfies statement (i) of Theorem~\ref{thm:Lukacs}. To that end,
let $l, n \in \mathbb{N}$. For every $k \geq k_0$ , we have that
\begin{align}
o(e_{2(k^3-k^2)} + \ldots + e_{2(k^3-k)} + e_{2k^3}) = 2^{2k^3} >
2^{k_0} = o(-x),
\end{align}
and so by Remark~\ref{rem:canon:ord},
\begin{align}
o(b_{k})= \max\{o(-x), o(e_{2(k^3-k^2)} + \ldots + e_{2(k^3-k)} +
e_{2k^3})\} = 2^{2k^3}.
\end{align}

Since $e_{k}\longrightarrow 0$, in the subgroup topology inherited
from $\mathbb{Q}/\mathbb{Z}$, $\{e_{k}\}$ is a $T$-sequence. Thus,~by
Theorem~\ref{thm:Lukacs}, there exists $M_1$ such that for every
$m \geq M_1$, $A[n] \cap A(l,m)_{\underline e} = \{0\}$. On the~other
hand, since $o(b_{k})= 2^{2k^3}$ for every $k\geq k_0$, and
$2(k+1)^3 - 2k^3 \longrightarrow \infty$,  by Lemma~\ref{lem:Tseq},
$\{b_{k}\}$ is also a~$T$-sequence. So, by Theorem~\ref{thm:Lukacs},
there exists $M_2$ such that for any
\mbox{$m \hspace{-2pt} \geq \hspace{-2pt} M_2$},~%
\mbox{$A[n]\hspace{-2.25pt} \cap\hspace{-2.25pt} A(l,m)_{\underline b}
\hspace{-2.25pt}= \hspace{-2.25pt} \{0\}\hspace{-0.75pt}$}.

Put $m_0 = \max\{M_1, M_2, 4l + n + k_0\}$. For any $m \geq m_0$,
one has that
\begin{align}
A[n] \cap A(l,m)_{\underline e}= A[n] \cap A(l,m)_{\underline b} =\{0\}.
\end{align}
Since
\begin{align}
A(l, 2m)_{\underline d} \subseteq A(l,m)_{\underline e} \cup
A(l,m)_{\underline b} \cup (A(l,m)_{\underline e}\backslash\{0\}+
A(l,m)_{\underline b}\backslash\{0\}),
\end{align}
it suffices to show that for every $m\geq m_0$,
\begin{align}
(A(l,m)_{\underline e}\backslash\{0\}+
A(l,m)_{\underline b}\backslash\{0\}) \cap A[n] = \emptyset.
\end{align}
Let \mbox{$z\in A(l,m)_{\underline e}\backslash\{0\}$} and
\mbox{$w =m_1b_{k_1}+ \cdots + m_h b_{k_h} \in
A(l,m)_{\underline b}\backslash\{0\}$} where \mbox{$0<\sum |m_i| \leq l$}
and \mbox{$m\leq k_1 <\cdots < k_h$}. There are $k_h+1$ summands in
\mbox{$y = e_{2(k_h^3 -k_h^2)} + \cdots + e_{2(k_h^3 -k_h)} +
e_{2k_h^3}$}. Moreover, the indices of every two consecutive
summands differ by $k_h$, and
\mbox{$k_h+1 \hspace{-2pt}>\hspace{-2pt} k_h \hspace{-2pt}> \hspace{-2pt}
m  \hspace{-2pt}\geq \hspace{-2pt} m_0 \hspace{-2pt}> \hspace{-2pt} 4l$}.
Since
$|m_h|\leq\sum |m_i|\leq l$, $y$ and $z$ satisfy the hypothesis of
Corollary~\ref{cor:canon:y-z}, and one obtains that
$o(m_hy + z) > 4^{k_h^3 - k_h^2 -l -1}$. Since $k_h > l \geq 1$,
\begin{align}
k_h^3 -k_h^2 -l -1 > k_h^3-k_h^2-k_h-1 \geq k_h^3-3k_h^2+3k_h-1=
(k_h-1)^3.
\end{align}
So, $o(m_hy + z) > 4^{(k_h^3 -k_h^2 -l -1)} > 4^{(k_h-1)^3}$.
Moreover, one has that \mbox{$k_0 < m_0-1 < k_h-1$}, because
\mbox{$m_0 = \max\{M_1, M_2, 4l + n + k_0\}$}. Thus,
\begin{align}
o(-m_hx) \leq 2^{k_0} < 4^{(k_h-1)^3} < o(m_hy + z).
\end{align}
Thus, $o(-m_h x) \neq o(m_hy + z)$, and by Remark~\ref{rem:canon:ord},
\begin{align}
o(m_h b_{k_h} + z) &= o((-m_h x) +( m_hy + z))
\\&= \max\{o(-m_hx),o(m_hy + z)\} = o(m_hy
+ z) > 4^{(k_h-1)^3}.
\end{align}
On the other hand,
\begin{align}
o(w - m_hb_{k_h}) & \leq o(b_{k_{h-1}})= 4^{k_{h-1}^3} \leq
4^{(k_h-1)^3} < o(m_hb_{k_h} + z).
\end{align}
By Remark~\ref{rem:canon:ord},
\begin{align}
o(w+z) & =o((w-m_hb_{k_h})+(m_hb_{k_h} + z)) \\
&= \max \{o(w-m_hb_{k_h}),o(m_hb_{k_h} + z)\} > 4^{(k_h-1)^3}
> 4^{(m_0-1)^3} > n.
\end{align}
Therefore, $A[n] \cap A(l,2m)_{\underline d} \hspace{-1.33pt}=
\hspace{-2.33pt} \{0\}$ for every
\mbox{$m \hspace{-2pt}\geq\hspace{-2pt} m_0$}. Hence, by
Theorem~\ref{thm:Lukacs}, $\{d_{k}\}$ is a $T\mbox{-}$sequence.

(b) Since every continuous character of $\mathbb{Z}(2^\infty)\{d_k\}$ is
also a continuous character of $\mathbb{Z}(2^\infty)$,
$\widehat{\mathbb{Z}(2^\infty)\{d_k\}}$ is contained in
$\mathbb{Z}_2 = \widehat{\mathbb{Z}(2^\infty)}$. By
the universal property of $\mathbb{Z}(2^\infty)\{d_k\}$,
\mbox{$\chi\in \widehat{\mathbb{Z}(2^\infty)}$}~is a~continuous
character of $\mathbb{Z}(2^\infty)\{d_k\}$ if and only if
\mbox{$\chi(d_k)\hspace{-2.6pt}\longrightarrow\hspace{-1.5pt} 0$},
that is,
\mbox{$\chi(b_k)\hspace{-2.7pt}\longrightarrow\hspace{-1.5pt} 0$} and
\mbox{$\chi(e_k) \hspace{-2.6pt}\longrightarrow\hspace{-1.5pt} 0$}.
By Theorem~\ref{thm:prel:ZP}(c),
\mbox{$\chi(e_k) \hspace{-2.6pt}\longrightarrow\hspace{-1.5pt}0$} holds if
and only if $\chi = m \chi_1$ for some $m \in \mathbb{Z}$. On the other
hand, since
\begin{equation}
0\leq \frac 1 {2^{2(k^3 -k^2)}} + \cdots + \frac 1 {2^{2(k^3-2k)}} +
\frac 1 {2^{2(k^3-k)}} +\frac 1 {2^{2k^3}} \leq \frac {k+1}{2^{2(k^3
-k^2)}} \longrightarrow 0,
\end{equation}
one has that $\chi_1(b_k) \longrightarrow -x$ in $\mathbb{T}$. Thus,
$\chi(b_k)= m\chi_1(b_k)\longrightarrow 0$ if and only if $-mx = 0$,
which means that \mbox{$x \hspace{-2pt}\in \hspace{-2pt} \ker\chi$}, and
\mbox{$o(x) \hspace{-2pt} = \hspace{-1.5pt} 2^{k_0}
\hspace{-2pt}\mid \hspace{-2pt} m$}. So,
\mbox{$\chi \hspace{-2pt}\in\hspace{-2pt} \widehat{\mathbb{Z}(2^\infty)}$} is
a continuous \mbox{character}~of~$\mathbb{Z}(2^\infty)\{d_k\}$~if and only
if \mbox{$\chi \hspace{-2pt}= \hspace{-2pt} m\chi_1$} for
\mbox{$m \hspace{-2pt}\in \hspace{-2pt} \mathbb{Z}$}
and \mbox{$2^{k_0} \hspace{-2pt}\mid \hspace{-2pt} m$}.
Therefore, the underlying group of $\widehat{\mathbb{Z}(2^\infty)\{d_k\}}$
is $2^{k_0} \mathbb{Z}$.

(c) It follows from the argument in part (b) that
\mbox{$x \hspace{-2pt}\in\hspace{-2pt} \ker\chi$}
for every continuous character~$\chi$~of
$\mathbb{Z}(2^\infty)\{d_k\}$, and so
\mbox{$x \hspace{-2pt}\in \hspace{-2pt} \bigcap\ker \chi \hspace{-2pt}
= \hspace{-1pt} \mathbf{n}(\mathbb{Z}(2^\infty)\{d_k\})$}.
Thus,  \mbox{$\langle x\rangle \hspace{-2pt} \subseteq \hspace{-2pt}
\mathbf{n}(\mathbb{Z}(2^\infty)\{d_k\})$}.
On the other hand, since $2^{k_0}\chi_1$ is a continuous character of
$\mathbb{Z}(2^\infty)\{d_k\}$, one has that
\begin{align}
\mathbf{n}(\mathbb{Z}(2^\infty)\{d_k\})=\bigcap\ker \chi \subseteq
\ker 2^{k_0}\chi_1 = \langle x\rangle.
\end{align}
Therefore,
\mbox{$\mathbf{n}(\mathbb{Z}(2^\infty)\{d_k\})
\hspace{-2pt}= \hspace{-2pt} \langle x\rangle $}, as desired.
\end{proof}

We proceed by introducing some notations prior to the proof of
Theorem~\ref{main:ctbl}. We denote by $\mathsf{Ab}$ and
$\mathsf{Ab(Haus)}$ the categories of abelian groups and abelian
Hausdorff topological groups, respectively (with the usual
morphisms). To abbreviate the notations, we introduce the following
class of abelian groups:
\begin{align}
\mathcal{A} & = \{ A \in \mathsf{Ab}\mid (\exists \tau)(
(A,\tau) \in \mathsf{Ab(Haus)} \wedge
(A,\tau) \text{ is almost maximally almost-periodic)}\}.
\end{align}
As indicated in Section~\ref{sect:prelim}, the proof of
Theorem~\ref{main:ctbl} is based on establishing certain algebraic
properties of the class $\mathcal{A}$.

\begin{theorem} \label{thm:proofs:absorb}
Let $B\in \mathsf{Ab}$, and let $A$ be a subgroup of $B$. If $A \in
\mathcal{A}$, then $B \in \mathcal{A}$.
\end{theorem}
\begin{proof}
Let \mbox{$\pi\colon B\rightarrow B/A$} denote the canonical
projection, and let $\tau_A$ be a Hausdorff group topology on $A$
such that $\mathbf{n}(A,\tau_A)$ is non-trivial and finite. Then,
$\tau_A$ can be extended to a Hausdorff group topology $\tau_B$ on
$B$ by defining the neighborhoods of $0$ with respect to $\tau_A$
and $B$ itself to be the neighborhoods of $0$ with respect to
$\tau_B$. It remains to be seen that
\mbox{$\mathbf{n}(A,\tau_A)=\mathbf{n}(B,\tau_B)$}. Let \mbox{$x\in
\mathbf{n}(A,\tau_A)$}. Since every continuous character \mbox{$\chi
\in \widehat{(B,\tau_B)}$} restricted to $A$ is a continuous
character of $A$, one has that \mbox{$\chi(x)=\chi_{|A}(x)=0$}.
Thus, \mbox{$x\in \mathbf{n}(B,\tau_B)$}. This shows that
\mbox{$\mathbf{n}(B,\tau_B)\subseteq \mathbf{n}(A,\tau_A)$}.
Conversely, let \mbox{$x\in \mathbf{n}(B,\tau_B)$}. Assume that
\mbox{$x\not\in A$}. Then, \mbox{$\pi(x)\neq 0$}. Since $A$ is an
open subgroup of~$B$, the quotient $B/A$ is a discrete abelian
group. Consequently, there exists a continuous character \mbox{$\psi
\colon B/A\rightarrow \mathbb{T}$} such that \mbox{$\psi(\pi(x))\neq
0$} (cf.~\cite[Theorem~39]{Pontr}). So, \mbox{$\psi\pi \in
\widehat{(B,\tau_B)}$} and \mbox{$\psi\pi(x) \neq 0$}, which
contradicts the assumption that \mbox{$x\in \mathbf{n}(B,\tau_B)$}.
This shows that \mbox{$x\in A$}. To conclude, we prove that
\mbox{$x\in \mathbf{n}(A,\tau_A)$}. Every continuous character
\mbox{$\psi\in \widehat{(A,\tau_A)}$} can be extended to a character
$\chi$ on $B$ (because $\mathbb{T}$ is injective). Since
\mbox{$\chi_{|A}=\psi$} is continuous on $A$, which is an open
subgroup of $B$, $\chi$ is continuous on $B$. Thus,
\mbox{$\chi\in\widehat{(B,\tau_B)}$}, and it follows that
\mbox{$\psi(x)=\chi(x)=0$}. Therefore, \mbox{$x\in
\mathbf{n}(A,\tau_A)$}. Hence,
\mbox{$\mathbf{n}(B,\tau_B)=\mathbf{n}(A,\tau_A)$}, which is
non-trivial and finite.
\end{proof}

\begin{proof}[Proof of Theorem~\ref{main:ctbl}.]
Let $B$ be an abelian group with an infinite torsion subgroup $A$.
Since $A$ is infinite and torsion, by Lemma~\ref{lem:infinite}, $A$
contains a subgroup that is isomorphic to either
$\mathbb{Z}(p^\infty)$ or \mbox{$\bigoplus\limits_{i=1}^\infty
C_{i}$}. Moreover, one has that \mbox{$\mathbb{Z}(p^\infty) \in
\mathcal{A}$} by Theorem~\ref{main:p2}, and
\mbox{$\bigoplus\limits_{i=1}^\infty C_{i} \in \mathcal{A}$} by
Theorem~\ref{thm:prel:dsf}. Therefore, it follows from
Theorem~\ref{thm:proofs:absorb} that \mbox{$A\in \mathcal{A}$}, and
hence one has that \mbox{$B\in\mathcal{A}$}, as desired.
\end{proof}

\begin{remark}
In the original version of this manuscript, it was only shown that
Theorem~\ref{main:ctbl} holds for abelian torsion groups $G$ where
$|G|=\aleph_0$ or $|G|> \mathfrak{c}$. I am grateful to the
anonymous referee for suggesting Theorem~\ref{thm:proofs:absorb},
which led to an improvement in the statement of
Theorem~\ref{main:ctbl}, as well as great simplification of its
proof.
\end{remark}
\section{Two open problems}

\label{sect:open}

Theorem~\ref{main:ctbl} naturally leads to the following problem.

\begin{problem}
Is there an infinite abelian group $E$ with a non-trivial torsion
subgroup that does not admit an almost maximally almost-periodic
group topology?
\end{problem}

\begin{discussion}
Theorem~\ref{main:ctbl} implies that if such an infinite abelian
group $E$ exists, then its torsion subgroup must be finite. At the
time this manuscript is being revised, a negative answer to this problem
was conjectured by Gabriyelyan (cf.~\cite[Theorem~5]{Gabriy}), but
his proof dated February 4, 2009, available on the ArXiv preprint
server, is incomplete. In Gabriyelyan's manuscript, a variation of
Theorem~\ref{thm:proofs:absorb} was also presented.

\end{discussion}

Theorem~\ref{main:ctbl} also raises another non-trivial question.

\begin{problem} \mbox{ }
\begin{myalphlist}

\item
Which abelian topological groups occur as the von Neumann radical of a
(Hausdorff) abelian topological group?

\item
Which abelian groups occur (algebraically) as the von Neumann radical of a
(Hausdorff) abelian topological group?

\end{myalphlist}
\end{problem}

\begin{discussion}
Due to the algebraic nature of this discussion, we are more
interested in part (b) of this problem. We put $\cong_a$ to denote
an isomorphism in $\mathsf{Ab}$. To abbreviate the notations, we
introduce the following class of abelian groups:
\begin{align}
\mathcal{B} & = \{ A \in  \mathsf{Ab} \mid\exists G \in \mathsf{Ab(Haus)},
\mathbf{n}(G)\cong_a A\}.
\end{align}
\vspace{-12pt}
\begin{myalphlist}
\item
By Theorem~\ref{thm:prel:ZP} (a) and (b), \mbox{$\mathbb{Z} \in
\mathcal{B}$} and \mbox{$\mathbb{Z}(p^\infty)\in\mathcal{B}$} for
every prime $p$.
\item
One has that \mbox{$\mathbb{R}\in\mathcal{B}$}, because $\mathbb{R}$
admits a minimally almost periodic Hausdorff group topology coarser
than the Euclidean topology of $\mathbb{R}$
(cf.~\cite[Theorem]{nienhuys}).
\item
Since every continuous character of $\mathbb{Q}$ can be extended
uniquely to a continuous character of $\mathbb{R}$, and every
continuous character of $\mathbb{R}$ restricted to $\mathbb{Q}$ is a
continuous character of $\mathbb{Q}$,
\mbox{$\mathbf{n}(\mathbb{R},\tau)\cap \mathbb{Q}
=\mathbf{n}(\mathbb{Q},\tau_{|\mathbb{Q}})$}. Thus, $\mathbb{Q} \in
\mathcal{B}$.
\item
If \mbox{$\{A_\alpha\}_{\alpha \in I}$} is an arbitrary family in
$\mathcal{B}$, then \mbox{$G \in \mathcal{B}$} for every abelian
group $G$ such that \mbox{$\bigoplus\limits_{\alpha \in I} A_\alpha
\subseteq G \subseteq \prod\limits_{\alpha \in I} A_\alpha$}; in
particular, \mbox{$\bigoplus\limits_{\alpha \in I} A_\alpha \in
\mathcal{B}$}, and \mbox{$\prod\limits_{\alpha \in I} A_\alpha \in
\mathcal{B}$}.
\item
Due to the algebraic structures of abelian groups, it follows from
the foregoing observations that the class $\mathcal{B}$ includes the
following types of abelian groups: free abelian groups, divisible
groups, direct sums of cyclic groups, bounded torsion groups, and
finitely generated groups.
\item
 We feel that we have not exploited the full strength of (d).
 Therefore, this task is left to another day.
\vskip 6pt
\end{myalphlist}
\end{discussion}

\section*{Acknowledgments}

This paper stems from my undergraduate summer research project in 2008,
under the supervision of Dr. G\'abor Luk\'acs. I wish to express my
gratitude to Dr. Luk\'acs for his patience and guidance throughout the
project, and for his extensive help with the preparation of this
manuscript.

I am thankful to Dr. Brendan Goldsmith for the helpful correspondence
concerning the structure of reduced $p$-groups.

We are grateful to Karen Kipper for her kind help in proof-reading this
paper for grammar and punctuation.

Last but not least, I would like to thank the anonymous referee for
her/his helpful comments and suggestions, which led to significant
improvements of the manuscript, not only in form, but also in substance.

{\footnotesize

\bibliography{notes,notes2,notes3}

\def\cprime{$'$} \def\cprime{$'$}
\begin{thebibliography}{10}

\bibitem{AHK}
M.~Ajtai, I.~Havas, and J.~Koml{\'o}s.
\newblock Every group admits a bad topology.
\newblock In {\em Studies in pure mathematics}, pages 21--34. Birkh\"auser,
  Basel, 1983.

\bibitem{DikMilTon}
D.~Dikranjan, C.~Milan, and A.~Tonolo.
\newblock A characterization of the maximally almost periodic abelian groups.
\newblock {\em J. Pure Appl. Algebra}, 197(1-3):23--41, 2005.

\bibitem{Fuchs}
L.~Fuchs.
\newblock {\em Infinite abelian groups. {V}ol. {I}}.
\newblock Pure and Applied Mathematics, Vol. 36. Academic Press, New York,
  1970.

\bibitem{Gabriy}
S.~Gabriyelyan.
\newblock On {$T$}-sequences and characterized subgroups.
\newblock {\em Preprint}, 2009.
\newblock ArXiv: 0902.0723v1.

\bibitem{HewRos}
E.~Hewitt and K.~A. Ross.
\newblock {\em Abstract harmonic analysis}.
\newblock Academic Press Inc., Publishers, New York, 1963.

\bibitem{GL9}
G.~Luk{\'a}cs.
\newblock Almost maximally almost-periodic group topologies determined by
  {$T$}-sequences.
\newblock {\em Topology Appl.}, 153(15):2922--2932, 2006.

\bibitem{NeuWig}
J.~von Neumann and E.~P. Wigner.
\newblock Minimally almost periodic groups.
\newblock {\em Ann. of Math. (2)}, 41:746--750, 1940.

\bibitem{nienhuys}
J.~W. Nienhuys.
\newblock A solenoidal and monothetic minimally almost periodic group.
\newblock {\em Fund. Math.}, 73(2):167--169, 1971/72.

\bibitem{Pontr}
L.~S. Pontryagin.
\newblock {\em Selected works. {V}ol. 2}.
\newblock Gordon \& Breach Science Publishers, New York, third edition, 1986.
\newblock Topological groups, Edited and with a preface by R. V. Gamkrelidze,
  Translated from the Russian and with a preface by Arlen Brown, With
  additional material translated by P. S. V. Naidu.

\bibitem{PZMono}
I.~Protasov and E.~Zelenyuk.
\newblock {\em Topologies on groups determined by sequences}, volume~4 of {\em
  Mathematical Studies Monograph Series}.
\newblock VNTL Publishers, L\cprime viv, 1999.

\bibitem{TkaYasch}
M.~Tkachenko and I.~Yaschenko.
\newblock Independent group topologies on abelian groups.
\newblock In {\em Proceedings of the International Conference on Topology and
  its Applications (Yokohama, 1999)}, volume 122, pages 425--451, 2002.

\bibitem{ZelProt}
E.~G. Zelenyuk and I.~V. Protasov.
\newblock Topologies on abelian groups.
\newblock {\em Izv. Akad. Nauk SSSR Ser. Mat.}, 54(5):1090--1107, 1990.

\end{thebibliography}

}

\bigskip
\noindent
Department of Mathematics\\
University of Manitoba\\
Winnipeg, MB, R3T 2N2\\
Canada \\
{\em e-mail: umnguyeb@cc.umanitoba.ca }

\end{document}